\documentclass[a4paper,11pt]{amsart}
\usepackage{amsmath,amsfonts,amssymb,amsthm,enumerate}
\usepackage[left=2.5cm,right=2.5cm,top=3cm,bottom=3cm,a4paper]{geometry}
\usepackage[onehalfspacing]{setspace}
\usepackage{mathrsfs}
\usepackage[pdftex]{graphicx}
\usepackage[all]{xy}
\usepackage{tikz}
\usepackage{float}
\usepackage{tikz-3dplot}

\usetikzlibrary{positioning}
\tikzset{>=stealth}

\theoremstyle{plain}
\newtheorem{thm}{Theorem}[section]
\newtheorem{lem}[thm]{Lemma}
\newtheorem{prop}[thm]{Proposition}
\newtheorem{cor}[thm]{Corollary}

\theoremstyle{definition}
\newtheorem{defn}[thm]{Definition}

\newtheorem{ques}[thm]{Question}

\newcommand{\OO}{\mathcal{O}}

\newcommand{\ceil}[1]{\lceil #1 \rceil}
\newcommand{\floor}[1]{\lfloor #1 \rfloor}

\title{Positivity of line bundles on general blow-ups of abelian surfaces}
\author{Sanghyeon Lee and Jaesun Shin}
\date{}

\address{Department of Mathematical Sciences, Seoul National University, GwanAkRo 1, Seoul 08826, Korea}

\email{tlrehrdl@snu.ac.kr}

\address{Department of Mathematical Sciences, KAIST, 291 Daehak-ro, Yuseong-gu, Daejeon 305-701, Korea}

\email{jsshin1991@kaist.ac.kr}

\begin{document}
\maketitle

\begin{abstract}
Let $(S,L_{S})$ be a polarized abelian surface, and let $M = c \cdot \pi^*L_S - \alpha \cdot \sum_{i=1}^r E_i$ be a line bundle on ${\rm Bl}_{r}(S)$, where $\pi:{\rm Bl}_{r}(S) \rightarrow S$ is the blow-up of $S$ at $r$ general points with exceptional divisors $E_{1},\dots,E_{r}$. In this paper, we provide a criterion for $k$-very ampleness of $M$. Also, we deal with the case when $S$ is an arbitrary surface of Picard number one with a numerically trivial canonical divisor.  
\end{abstract}

\begin{section} {Introduction}

Let $(S,L_{S})$ be a complex polarized surface, and let $\pi:{\rm Bl}_{r}(S) \rightarrow S$ be the blow-up of $S$ at $r$ general points $p_{1},\dots,p_{r}$ with exceptional divisors $E_{1},\dots,E_{r}$. Line bundles of the form 
\begin{align*}
M = c \cdot \pi^*L_S - \alpha \cdot \sum_{i=1}^r E_i
\end{align*}
have been addressed by many authors with respect to different properties. Here we are interested in their positivity. The positivity of $M$ provides positivity data of $L_{S}$ along $r$ general points $p_{1},\dots,p_{r}$.

A classical way to study the positivity of $M$ is to check its global generation and its very ampleness. The notions of the global generation and the very ampleness were generalized in several ways. Here we focus on the $k$-very ampleness, which generalizes in a natural way two classical notions:

\begin{defn} \label{k-very ampleness}
A line bundle $L$ on a smooth projective variety $S$ is $k$-very ample if the restriction map $H^0(S,L) \to H^0(S,L|_{Z})$ is surjective for any $0$-dimensional subscheme $Z \subset S$ with length $k+1$.
\end{defn}

In fact, $L$ is $0$-very ample if and only if it is generated by global sections, and $L$ is $1$-very ample if and only if it is very ample. Geometrically, if $L$ is very ample and embeds $S$ in $\mathbb{P}^{h^{0}(L)-1}$, then $L$ is $k$ -very ample if and only if $S$ has no $(k+1)$-secant $(k-1)$ planes. 

Due to its importance, there are many studies about local positivity of line bundle $L$ on an algebraic surface $S$ of various types. When $S$ is an arbitrary smooth projective surface, K\"uchle \cite{K2} proved that lines bundles of the form $M=c\cdot \pi^* L - \alpha \cdot \sum_{i=1}^r E_i$ are ample when $c \geq 3$, $\alpha=1$ and $M^2 > 0$. Tutaj-Gasi\'nska \cite{T2} and Szemberg-Tutaj-Gasi\'nska \cite{ST2} carried out studies for the case where $S$ is a ruled surface. When $S=\mathbb{P}^2$, d'Almeida-Hirschowitz \cite{DH} obtained a complete answer for the very ampleness of a line bundle $c \cdot \pi^{*} \OO_{\mathbb{P}^{2}}(1)-\sum_{i=1}^{r}E_{i}$, and Hanumanthu (\cite{H}) studied the global generation, ampleness, and very ampleness of a line bundle $c \cdot \pi^{*} \OO_{\mathbb{P}^{2}}(1)-\sum_{i=1}^{r}m_{i}E_{i}$. Moreover, Szemberg-Gasi\'nska \cite{ST3} studied a criterion for $k$-very ampleness of the line bundle $M$. In general, for a projective space $\mathbb{P}^n$, d'Almeida-Hirschowski \cite{DH} and Coppens \cite{C} studied sufficient conditions for very ampleness of $M$. Finally, Tutaj-Gasi\'nska \cite{T} dealt with the $k$-very ampleness of $M$ when $S$ is an abelian surface of Picard number $1$, where $c=1$ and $\alpha=k$. Szemberg-Tutaj-Gasi\'nska \cite{ST1} worked on the very ampleness of $M$, where $c=1$ and $\alpha=1$, for an arbitrary Picard number. 

From now on, let $(S,L_{S})$ be a polarized abelian surface, and let $\rho(S)$ be the Picard number of $S$ unless otherwise specified. In this paper, we continue the investigations initiated in \cite{T} and \cite{ST1} to find a satisfactory result of positivity for $M$. To be more specific, our aim is to study $k$-very ampleness of $M$ in a fully generalized setting: arbitrary $c$, $\alpha \ge k$, and $\rho(S)$. Note that the condition $\alpha \ge k$ is necessary for the $k$-very ampleness of $c \cdot \pi^{*}L_{S}-\alpha \cdot \sum_{i=1}^{r}E_{i}$ since $(c \cdot \pi^{*}L_{S}-\alpha \cdot \sum_{i=1}^{r}E_{i}).E_{j}=\alpha$ for each $j=1,\dots,r$. In this direction, we prove the following results:

\begin{thm} (=Theorem \ref{thm:higher Picard number one}) \label{thm:introduction 1}
Let $(S,L_{S})$ be a polarized abelian surface of type $(1,d)$ with $\rho(S)=1$. Let $\pi:{\rm Bl}_{r}(S) \rightarrow S$ be the blow-up of $S$ at $r$ general points with $1 \le r \le \frac{2d-(4\ceil{\frac{k}{c}}+5)}{(\ceil{\frac{\alpha}{c}}+1)^{2}}$ for integers $\alpha \ge k \ge 0$ and $c \in \mathbb{N}$. Then the line bundle $M_{c,\alpha}=c \cdot \pi^{*}L_{S}- \alpha \cdot \sum_{i=1}^{r}E_{i}$ is $k$-very ample. 
\end{thm}

\begin{thm} (=Theorem \ref{thm:any Picard number}) \label{thm:introduction 2}
Let $(S,L)$ be a polarized abelian surface of type $(1,d)$. Let $\pi:{\rm Bl}_{r}(S) \rightarrow S$ be the blow-up of $S$ at $2 \le r < \frac{2d}{(\ceil{\frac{\alpha}{c}}+1)^{2}}-2$ general points.
\begin{enumerate}[(1)]
\item Assume that $S$ contains no elliptic curve of degree $\le \ceil{\frac{\alpha}{c}} +1$ with respect to $L$. Then the line bundle $M_{c,\alpha}=c \cdot \pi^{*}L_{S}-\alpha \cdot \sum_{i=1}^{r} E_{i}$ is globally generated when $\alpha \ge 0$.
\item Assume that $S$ contains no elliptic curve of degree $\le \ceil{\frac{\alpha}{c}} +2$ with respect to $L$. Then the line bundle $M_{c,\alpha}=c \cdot \pi^{*}L_{S}-\alpha \cdot \sum_{i=1}^{r} E_{i}$ is very ample when $\alpha \ge 1$. 
\end{enumerate}
\end{thm}

\begin{thm} (=Theorem \ref{thm:k-very ampleness on abelian surface}) \label{thm:introduction 3}
Let $(S,L)$ be a polarized abelian surface of type $(1,d)$. Let $\pi:{\rm Bl}_{r}(S) \rightarrow S$ be the blow-up of $S$ at $2 \le r <\frac{2d}{(\ceil{\frac{\alpha}{c}}+\ceil{\frac{k}{c}}+1)^{2}}-2$ general points for integers $\alpha \ge k \ge 0$. Assume that $S$ contains no elliptic curve of degree $\le \ceil{\frac{\alpha}{c}}+\ceil{\frac{k}{c}}+1$ with respect to $L$. Then $M=c\cdot\pi^{*}L-\alpha \cdot \sum_{i=1}^{r}E_{i}$ is $k$-very ample. 
\end{thm}

In particular, when restricted to the cases of \cite{T} and \cite[Theorem 1]{ST1}, Theorem \ref{thm:introduction 1}, Theorem \ref{thm:introduction 2}, and Theorem \ref{thm:introduction 3} are improved relative to the previous results.

If $S$ is a simple abelian surface, then the condition for elliptic curves is not needed. On most abelian surfaces, however, the condition on the lower bound of degrees of elliptic curves may be a nuisance. Regarding this, we also give a criterion for $k=0,1$-very ampleness without the condition for elliptic curves (cf. Corollary \ref{cor:application 1} and Corollary \ref{cor:application 2}). 

In addition, we deal with the case when $S$ is an arbitrary surface of Picard number one with a numerically trivial canonical divisor: 

\begin{prop} (=Proposition \ref{prop:picard number 1 case}) \label{prop:introduction 4}
Let $S$ be a surface of $\rho(S)=1$ with a numerically trivial canonical divisor and $L$ the ample generator. Let $\pi:{\rm Bl}_{r}(S) \rightarrow S$ be the blow-up at $2 \le r \le \frac{L^{2}}{(\ceil{\frac{\alpha}{c}}+2)^{2}}$ general points for integers $\alpha \ge k \ge 0$ and $c \in \mathbb{N}$. Then the line bundle $M=c \cdot \pi^{*}L-\alpha \cdot \sum_{i=1}^{r}E_{i}$ is $k$-very ample. 
\end{prop}

The main tool used in our study is a Reider-type criterion \cite{BS} for $k$-very ampleness, which is proved by Beltrametti-Sommese, and the results for Seshadri constants of abelian surfaces and those of some projective varieties (\cite{B, BS1, F, K, N, Sz}). 

Concerning the organization of the paper, we begin in Section \ref{section:2} by fixing notation and collecting useful facts about Seshadri constants and a Reider-type criterion for $k$-very ampleness. In Section \ref{section:3}, we prove Theorem \ref{thm:introduction 1}. Using a similar method, we also give a criterion for $k$-very ampleness on general blow-ups of surfaces of Picard number $1$ with a numerically trivial canonical divisor (Proposition \ref{prop:introduction 4}). Finally, Section \ref{section:4} is devoted to the proofs of Theorem \ref{thm:introduction 2} and Theorem \ref{thm:introduction 3} as well as to their applications.  \\

\textbf{Acknowledgements.}
Research of the second author supported by NRF (National Research Foundation of Korea) Grant funded by the Korean Government (NRF-2016-Fostering Core Leaders of the Future Basic Science Program/Global Ph.D. Fellowship Program). 


\end{section}

\begin{section}{Preliminaries} \label{section:2}

\begin{subsection} {Reider type criterion}

For the discussion from now on, we recall the concept of $k$-very ampleness, which is a generalization of global generation and very ampleness:

\begin{defn}
Let $X$ be a smooth projective variety and $Z$ a $0$-dimensional subscheme of $X$ of length $k+1$ (,i.e. $h^{0}(\OO_{Z})=k+1$). A line bundle $L$ on $X$ is $k$-very ample at $X$ if the restriction map 
\begin{align*}
H^{0}(X,L) \rightarrow H^{0}(X,L \otimes \OO_{Z})
\end{align*}
is surjective. If $L$ is $k$-very ample at every $0$-dimensional subscheme of length $k+1$, we say $L$ is $k$-very ample on $X$. 
\end{defn}

It is immediate that a line bundle is $0$-very ample if and only if it is globally generated, and it is $1$-very ample if and only if it is very ample. \\

We also comment the Reider's celebrated result on global generation and very ampleness of line bundles (\cite{R}):

\begin{thm} \label{thm:Reider}
Let $X$ be a smooth projective surface, and let $N$ be a nef line bundle on $X$. 
\begin{enumerate}[(i)]
\item If $N^{2} \ge 5$, then either $K_{X}+N$ is globally generated or there exists an effective divisor $D$ on $X$ that satisfies one of the numerical conditions:
\begin{enumerate}[(1)]
\item $D.N=0$ and $D^{2}=-1$.
\item $D.N=1$ and $D^{2}=0$. 
\end{enumerate}
\item If $N^{2} \ge 10$, then either $K_{X}+N$ is very ample or there exists an effective divisor $D$ on $X$ that satisfies one of the numerical conditions:
\begin{enumerate}[(1)]
\item $D.N=0$ and $D^{2}=-1$ or $-2$. 
\item $D.N=1$ and $D^{2}=0$ or $-1$. 
\item $D.N=2$ and $D^{2}=0$.
\end{enumerate}
\end{enumerate}
\end{thm}

In the sequel, we will also make use of the following criterion proved by Beltrametti and Sommese (\cite{BS}). It is a Reider type criterion for higher order embeddings. 

\begin{thm} \label{thm:criterion}
Let $X$ be a smooth projective surface, and let $Z$ a $0$-dimensional subscheme of $X$ of length $k+1$. Let $L$ be a nef line bundle on $X$ such that $L^{2} \ge 4k+5$. Then either $K_{X}+L$ is $k$-very ample at $Z$ or there exists an effective divisor $D$ on $X$ satisfying the following properties:
\begin{enumerate}[(i)]
\item $L-2D$ is $\mathbb{Q}$-effective.
\item $D$ contains $Z$ and $K_{X}+L$ is not $k$-very ample at $Z$. 
\item $(L.D)-k-1 \le D^{2} < \frac{1}{2}(L.D)<k+1$.
\end{enumerate}
\end{thm}

\end{subsection}

\begin{subsection}{Seshadri constants}
The (multi-poiont) Seshadri constant of a nef and big line bundle $L$ at $x_{1}, \dots, x_{r}$ is the real number
\begin{align*}
\epsilon(L;x_{1},\dots,x_{r})=\sup\{a \ge 0 \text{ } | \text{ } \pi^{*}L-a \cdot \sum_{i=1}^{r}E_{i} \text{ is nef on $\tilde{X}$}\},
\end{align*}
where $\pi:\tilde{X} \rightarrow X$ is the blow-up of a smooth projective variety $X$ at $x_{1},\dots,x_{r} \in X$ (cf. \cite[Definition 1.5 and 1.8]{BDHKKSS}). Customary, for the Seshadri constant of a nef line bundle $L$ on $r$ very general points, we write briefly $\epsilon(L;r)$. We recall some crucial facts about Seshadri constants for later use. 

\begin{lem} (\cite{EKL}) \label{lem:general}
Let $X$ be a smooth projective variety, and let $N$ be a nef line bundle on $X$. Then for any $\delta \in \mathbb{R}_{>0}$, there exists a non-empty open subset $U_{\delta} \subseteq X$ such that $\epsilon(L;x_{1},\dots,x_{r}) \ge \epsilon(L;r)-\delta$ for any $\{x_{1},\dots,x_{r}\} \subseteq U_{\delta}$. 
\end{lem}

The following criterion is due to K\"uchle (\cite{K}). 

\begin{prop} \label{prop:multi lower bound}
Let $L$ be a nef and big line bundle on a smooth projective $n$-fold. For $r \ge 2$, we have
\begin{align*}
\epsilon(L;r) \ge \min\{\epsilon(L;1), \frac{\sqrt[n]{L^{n}}}{2}, \frac{\sqrt[n]{L^{n}(r-1)^{n-1}}}{r}\}.
\end{align*}
\end{prop}

Note that on abelian surface $S$, $\epsilon(L;1)=\epsilon(L,x)$ for any $x \in S$. From \cite{BS1}, the following lower bound can be obtained: 

\begin{prop} \label{prop:one lower bound}
Let $(S,L_{S})$ be a polarized abelian surface of type $(1,d)$. Then 
\begin{align*}
\epsilon(L;1) \ge \min\{\epsilon_{0},\frac{\sqrt{7}}{2} \sqrt{d}\},
\end{align*}
where $\epsilon_{0}$ is the minimal degree of an elliptic curve in $S$ with respect to $L_{S}$. 
\end{prop}

Moreover, if $(S, L_{S})$ is a polarized abelian surface of type $(1,d)$ with Picard number $1$, its Seshadri constant at $r$ general points is sufficiently close to its maximal value. 

\begin{prop} (\cite[Corollary 2.6]{F}) \label{prop:lower bound Picard number one}
Let $(S,L_{S})$ be a polarized abelian surface of type $(1,d)$ with $\rho(S)=1$, and let $x_{1},\dots,x_{r}$ be $r$ general points of $X$. Then:
\begin{enumerate}[(i)]
\item If $\sqrt{\frac{2d}{r}} \in \mathbb{Q}$, then $\epsilon(L_{S};x_{1},\dots,x_{r})=\sqrt{\frac{2d}{r}}$. 
\item If $\sqrt{\frac{2d}{r}} \notin \mathbb{Q}$, then $\epsilon(L_{S};x_{1},\dots,x_{r}) \ge 2d \cdot \frac{k_{0}}{l_{0}}=\sqrt{1-\frac{1}{l_{0}^{2}}} \sqrt{\frac{L_{S}^{2}}{r}}$, where $(l_{0},k_{0})$ is the primitive solution of Pell's equation $l^{2}-2rdk^{2}=1$. 
\end{enumerate}
\end{prop}

\end{subsection}

\end{section}

\begin{section}{Positivity on general blow-ups of surfaces with $\rho(S)=1$.} \label{section:3}

\begin{subsection}{Abelian surfaces with $\rho(S)=1$.}

Let $(S,L_{S})$ be a polarized abelian surface of type $(1,d)$ with $\rho(S)=1$, and let $\pi:{\rm Bl}_{r}(S) \rightarrow S$ be the blow-up of $S$ at $r$ general points $p_{1},\dots,p_{r}$. Our first result is about the $k$-very ampleness of $M_{c,\alpha}=c \cdot \pi^{*}L_{S}-\alpha \cdot \sum_{i=1}^{r}E_{i}$. 

\begin{lem} \label{lem:0 Picard number one}
Let $(S,L_{S})$ be a polarized abelian surface of type $(1,d)$ with $\rho(S)=1$. Let $\pi:{\rm Bl}_{r}(S) \rightarrow S$ be the blow-up of $S$ at $1 \le r \le \frac{2d-5}{{(\ceil{\frac{\alpha}{c}}+1)}^{2}}$ general points for $\alpha \in \mathbb{Z}_{\ge 0}$ and $c \in \mathbb{N}$. Then the line bundle $M_{c,\alpha}=c \cdot \pi^{*}L_{S}-\alpha \cdot \sum_{i=1}^{r} E_{i}$ is globally generated. 
\end{lem}

\begin{proof}
Let $N_{c,\alpha}=M_{c,\alpha}-K_{{\rm Bl}_{r}(S)}=c \cdot \pi^{*}L_{S}-(\alpha+1) \cdot \sum E_{i}$. First, we consider the case when $c=1$. Note that $N_{c,\alpha}^{2}=2d-r(\alpha+1)^{2} \ge 5$. Moreover, $N_{c,\alpha}$ is ample: in fact, $\epsilon(L_{S};r) \ge \sqrt{1-\frac{1}{l_{0}^{2}}} \sqrt{\frac{L_{S}^{2}}{r}}$, where $l_{0}$ is as in Proposition \ref{prop:lower bound Picard number one}-$(ii)$. Since $l_{0}^{2}=1+2rdk_{0}^{2} \ge 1+2rd$, 
\begin{align*}
\sqrt{1-\frac{1}{l_{0}^{2}}} \sqrt{\frac{L_{S}^{2}}{r}} \ge \sqrt{\frac{2rd}{1+2rd}}{\sqrt{\frac{2d}{r}}} =\frac{2d}{\sqrt{1+2rd}}.
\end{align*}
Since $\frac{2d}{\sqrt{1+2rd}}$ is a decreasing function on $r$, it is enough to check the inequality $\frac{2d}{\sqrt{1+2rd}}>\alpha+1$ for maximal $r=\frac{2d-5}{{(\alpha+1)}^{2}}$. By our assumption on $d$, it is easy to see that $\frac{2d}{\sqrt{(\alpha+1)^{2}+4d^{2}-10d}}>1$, that is, $\epsilon(L_{S};r)>\alpha +1$. Combining Lemma \ref{lem:general} with the fact that $\pi$ is the blow-up at $r$ general points, $N_{c,\alpha}$ is ample. \

Now, we are ready to use Theorem \ref{thm:Reider}. Suppose $M_{c,\alpha}$ is not globally generated. Then there exists an effective divisor $D$ on ${\rm Bl}_{r}(S)$ such that $D.N_{c,\alpha}=0$, $D^{2}=-1$, or $D.N_{c,\alpha}=1$, $D^{2}=0$. Since $N_{c,\alpha}$ is ample, there exists no effective divisor $D$ such that $D.N_{c,\alpha}=0$. So $D$ should satisfy $D.N_{c,\alpha}=1$ and $D^{2}=0$. Since $D^{2}=0$, $D$ should be of the form $\pi^{*}D_{S}-\sum a_{i}E_{i}$, where $D_{S}$ is a nontrivial effective divisor on $S$. Since $\rho(S)=1$, $D_{S}=aL_{S}$ for some $a \in \mathbb{N}$. The two numerical conditions give that $D^{2}=2a^{2}d-\sum a_{i}^{2}=0$ and $D.N_{c,\alpha}=2ad-(\alpha+1) \cdot \sum a_{i}=1$. So $(2ad-1)^{2}=(\alpha+1)^{2}(\sum a_{i})^{2} \le (\alpha+1)^{2} \cdot r \sum a_{i}^{2}$, that is, 
\begin{align*}
4a^{2}d^{2}-4ad \le (2ad-1)^{2} \le (\alpha+1)^{2}r \cdot 2a^{2}d.
\end{align*}
Dividing both sides by $2ad$, we obtain $2ad-2 \le r(\alpha+1)^{2} \cdot a$. Since $r \le \frac{2d-5}{{(\alpha+1)}^{2}}$, we have
\begin{align*}
2ad-2 \le r(\alpha+1)^{2} \cdot a \le (2d-5)a,
\end{align*}
i.e. $5a \le 2$, which is a contradiction. Thus $M_{c,\alpha}$ is globally generated. 

Next, we consider an arbitrary $c\in \mathbb{N}$. Let $\alpha = c\floor{\frac{\alpha}{c}} + \alpha'$, where $0\le \alpha' < c$. Then we have $\alpha = (c-\alpha')\floor{\frac{\alpha}{c}}+\alpha'\ceil{\frac{\alpha}{c}}$, so we can write
\begin{align*}
M_{c,\alpha}=(\pi^*L_S-\floor{\frac{\alpha}{c}}\cdot\sum_{i=1}^{r}E_{i})^{\otimes c-\alpha'}\otimes (\pi^*L_S-\ceil{\frac{\alpha}{c}}\cdot\sum_{i=1}^{r}E_{i})^{\otimes \alpha'}. 
\end{align*}
By the proof of $c=1$ case, $\pi^*L_S-\floor{\frac{\alpha}{c}}\cdot\sum_{i=1}^{r}E_{i}$ and $\pi^*L_S-\ceil{\frac{\alpha}{c}}\cdot\sum_{i=1}^{r}E_{i}$ are globally generated, i.e. $0$-very ample. Thus, by \cite[Theorem 1.1]{HTT}, $M_{c,\alpha}$ is globally generated. 
\end{proof}

Now, we can derive the following result on higher embeddings of a polarized abelian surface with Picard number one, which is a generalization of \cite[Theorem 2]{ST1} and \cite[Theorem 7]{T}. Note that the condition $\alpha \ge k$ is necessary for the $k$-very ampleness of $c \cdot \pi^{*}L_{S}-\alpha \cdot \sum_{i=1}^{r}E_{i}$ since $(c \cdot \pi^{*}L_{S}-\alpha \cdot \sum_{i=1}^{r}E_{i}).E_{j}=\alpha$ for each $j=1,\dots,r$. 

\begin{thm} \label{thm:higher Picard number one}
Let $(S,L_{S})$ be a polarized abelian surface of type $(1,d)$ with $\rho(S)=1$. Let $\pi:{\rm Bl}_{r}(S) \rightarrow S$ be the blow-up of $S$ at $r$ general points such that $1 \le r \le \frac{2d-(4\ceil{\frac{k}{c}}+5)}{(\ceil{\frac{\alpha}{c}}+1)^{2}}$ for integers $\alpha \ge k \ge 0$ and $c \in \mathbb{N}$. Then the line bundle $M_{c,\alpha}=c \cdot \pi^{*}L_{S}- \alpha \cdot \sum_{i=1}^{r}E_{i}$ is $k$-very ample. 
\end{thm}

Our main tools are the Reider type criterion (Theorem \ref{thm:criterion}) and Proposition \ref{prop:lower bound Picard number one}. From now on, we write $M$ instead of $M_{c,\alpha}$ by fixing $c$ and $\alpha$ if there is no confusion. Let $L$ be a line bundle numerically equivalent to $M-K_{{\rm Bl}_{r}(S)}$, that is, $L \equiv c \cdot \pi^{*}L_{S}-(\alpha+1) \cdot \sum_{i=1}^{r}E_{i}$. 

\begin{lem} \label{lem:4k+5}
When $c=1$, $L$ is ample with $L^{2} \ge 4k+5$ for any $r$, $d$, $k$, and $\alpha$ in Theorem \ref{thm:higher Picard number one}. 
\end{lem}

\begin{proof}
It is obvious that $L^{2} \ge 4k+5$, so we focus on the ampleness of $L$. It is sufficient to show that $\epsilon(L_{S};r)>\alpha+1$. In fact, we claim that 
\begin{align*}
\epsilon(L_{S};r)>\frac{d(\alpha+1)}{d-k}. 
\end{align*}
Then applying Lemma \ref{lem:general}, we obtain the ampleness of $L$. By Proposition \ref{prop:lower bound Picard number one}, we need to check 
\begin{align*}
\sqrt{1-\frac{1}{l_{0}^{2}}} \sqrt{\frac{L_{S}^{2}}{r}} > \frac{d(\alpha+1)}{d-k},
\end{align*}
where $l_{0}$ is defined as in Proposition \ref{prop:lower bound Picard number one}-$(ii)$. Since $l_{0}^{2}=1+2rdk_{0}^{2} \ge 1+2rd$ and $\sqrt{1-\frac{1}{l_{0}^{2}}}$ is an increasing function for $l_{0}^{2}$, it is enough to check it for $l_{0}^{2}=1+2rd$. Since $L_{S}^{2}=2d$, it is equivalent to 
\begin{align*}
\frac{1}{\sqrt{1+2rd}}>\frac{\alpha+1}{2d-2k}.
\end{align*}
Finally, since $\frac{1}{\sqrt{1+2rd}}$ is a decreasing function for $r$, it is enough to check for maximal $r$, that is, 
\begin{align*}
\frac{1}{\sqrt{(\alpha+1)^{2}+2d(2d-4k-5)}}>\frac{1}{2d-2k}.
\end{align*}
Now, it is easy to check that this inequality holds since 
\begin{align*}
(4d^{2}-8dk+4k^{2})-(\alpha^{2}+2\alpha+1+4d^{2}-8kd-10d)=10d+4k^{2}-(\alpha+1)^{2}>0.
\end{align*}
\end{proof}

\begin{proof}[Proof of Theorem \ref{thm:higher Picard number one}]
For $k=0$, the conclusion is a special case of Lemma \ref{lem:0 Picard number one}, so we may assume that $k \ge 1$. We first deal with the case when $c=1$. By Lemma \ref{lem:4k+5}, it is enough to show the non-existence of an effective divisor $D$ on ${\rm Bl}_{r}(S)$ satisfying Theorem \ref{thm:criterion}-$(iii)$. If $D=\sum a_{i}E_{i}$ with $a_{i} \ge 0$, $L.D-(k+1) \le D^{2}$ implies $\alpha<k$, which is a contradiction. \

So set $D=\pi^{*}D_{S}-\sum m_{i}E_{i}$. In the proof of Lemma \ref{lem:4k+5}, we already proved that $\epsilon(L_{S};x_{1},\dots,x_{r})>\frac{d(\alpha+1)}{d-k}$ for general $x_{1},\dots,x_{r}$. Since $D$ is effective, 
\begin{align*}
(\pi^{*}L_{S}-\frac{d(\alpha+1)}{d-k} \cdot \sum E_{i}.D)>0 \text{, that is, } (\alpha+1) \cdot \sum m_{i}<\frac{(L_{S}.D_{S})(d-k)}{d}.
\end{align*}
Hence, we obtain $L.D=L_{S}.D_{S}-(\alpha+1) \cdot \sum m_{i} > (L_{S}.D_{S}) \cdot \frac{k}{d}$. Since $\rho(S)=1$, $D_{S}=aL_{S}$ for some $a \in \mathbb{N}$. So the above inequalities, combined with Theorem \ref{thm:criterion}-$(iii)$, give only one possibility:
\begin{enumerate}
\item[] $a=1$, i.e. $D_{S}=L_{S}$ and $L.D=2k+1$.
\end{enumerate}

Since $L.D=2k+1$, $D^{2}=k$ holds, so that $\sum m_{i}^{2}=2d-k$ and $\sum m_{i}=\frac{2d-2k-1}{\alpha+1}$. From $L.D<2k+2$, we obtain:
\begin{align*}
4d^{2}=(L_{S}.D_{S})^{2} &<4(k+1)^{2}+4(k+1)(\alpha+1)\sum m_{i}+(\alpha+1)^{2}(\sum m_{i})^{2} \\
& \le 4(k+1)^{2}+4(k+1)(2d-2k-1)+(\alpha+1)^{2}r \cdot \sum m_{i}^{2} \\
& \le 4(k+1)^{2}+4(k+1)(2d-2k-1)+(2d-4k-5)(2d-k) \\
&=4d^{2}-2d-2kd+k,
\end{align*}
i.e. $2d(k+1) < k$, which excludes the possibility. Hence we conclude the result for $c=1$. 

For an arbitrary $c\in \mathbb{N}$, let $\alpha = c\floor{\frac{\alpha}{c}} + \alpha'$, where $0\le \alpha' < c$. Since $\alpha = (c-\alpha')\floor{\frac{\alpha}{c}}+\alpha'\ceil{\frac{\alpha}{c}}$, we can write 
\begin{align*}
M_{c,\alpha}=(\pi^*L_S-\floor{\frac{\alpha}{c}}\cdot\sum_{i=1}^{r}E_{i})^{\otimes c-\alpha'}\otimes (\pi^*L_S-\ceil{\frac{\alpha}{c}}\cdot\sum_{i=1}^{r}E_{i})^{\otimes \alpha'}.
\end{align*}
Next, let $k=c\floor{\frac{k}{c}}+k'$, that is, $k = (c-k')\floor{\frac{k}{c}}+k'\ceil{\frac{k}{c}}$. Since $\alpha \ge k$, either $\alpha' \ge k'$ or $\floor{\frac{\alpha}{c}} \ge \floor{\frac{k}{c}}+1$ holds.

If $\alpha'\ge k'$, then by the proof of $c=1$ case, $\pi^*L_S-\floor{\frac{\alpha}{c}}\cdot\sum_{i=1}^{r}E_{i}$ is $\floor{\frac{k}{c}}$-very ample, and $\pi^*L_S-\ceil{\frac{\alpha}{c}}\cdot\sum_{i=1}^{r}E_{i}$ is $\ceil{\frac{k}{c}}$-very ample. Then by \cite[Theorem 1.1]{HTT}, we have that $M_{c,\alpha}$ is $((c-c')\floor{\frac{k}{c}}+c'\ceil{\frac{k}{c}})$-very ample. Since $\alpha' \ge k'$, we obtain $((c-\alpha')\floor{\frac{k}{c}}+\alpha'\ceil{\frac{k}{c}})\ge ((c-k')\floor{\frac{k}{c}}+k'\ceil{\frac{k}{c}})=k$. Therefore $M_{c,\alpha}$ is $k$-very ample.

If $\floor{\frac{\alpha}{c}} \ge \floor{\frac{k}{c}}+1\ge \ceil{\frac{k}{c}}$, then $\pi^*L_S-\floor{\frac{\alpha}{c}}\cdot\sum_{i=1}^{r}E_{i}$ is $\floor{\frac{k}{c}}$ and $\pi^*L_S-\ceil{\frac{\alpha}{c}}\cdot\sum_{i=1}^{r}E_{i}$ are both $\ceil{\frac{k}{c}}$-ample. Therefore, again by \cite[Theorem 1.1]{HTT}, $M_{c,\alpha}$ is $c\ceil{\frac{k}{c}}$-very ample. Since $c\ceil{\frac{k}{c}}\ge k$, the conclusion follows. 
\end{proof}

\end{subsection}

\begin{subsection}{Surfaces of Picard number one with a numerically trivial canonical divisor.} 

The proof of Theorem \ref{thm:higher Picard number one} is also useful for analyzing higher embeddings of an arbitrary surface with Picard number one. As a special case, we observe the $k$-very ampleness of a polarized surface $(S,L)$ of $\rho(S)=1$ with the numerically trivial canonical divisor. 

\begin{prop}\label{prop:picard number 1 case}
Let $S$ be a surface of $\rho(S)=1$ with the numerically trivial canonical divisor, and let $L$ be the ample generator. Let $\pi:{\rm Bl}_{r}(S) \rightarrow S$ be the blow-up at $2 \le r \le \frac{L^{2}}{(\ceil{\frac{\alpha}{c}}+2)^{2}}$ general points for integers $\alpha \ge k \ge 0$ and $c \in \mathbb{N}$. Then the line bundle $M=c \cdot \pi^{*}L-\alpha \cdot \sum_{i=1}^{r}E_{i}$ is $k$-very ample. 
\end{prop}

\begin{proof}
First, we assume that $c=1$. Note that $K_{{\rm Bl}_{r}(S)} \equiv \sum E_{i}$ since $K_{S} \equiv \OO_{S}$. Let $L^{2}=l$ and $N=M-K_{{\rm Bl}_{r}(S)} \equiv c \cdot \pi^{*}L-(\alpha+1) \cdot \sum E_{i}$. We claim that $N$ is ample with $N^{2} \ge 4k+5$ since $N^{2} \ge l-r(\alpha+1)^{2} \ge \frac{2\alpha+3}{(\alpha+2)^{2}}l \ge 4\alpha+6 >4k+5$. We focus on the proof of the ampleness of $N$. 
For $r$ general points $x_{1},\dots,x_{r} \in S$, we prove the inequality 
\begin{align}
\epsilon(L;x_{1},\dots,x_{r}) > \frac{l(\alpha+1)}{l-k}.
\end{align}
For the proof, note that $\epsilon(L;r) \ge \lfloor \sqrt{\frac{l}{r}} \rfloor \ge \alpha+2$ by \cite[Theorem 3.2]{Sz}. So it is sufficient to show that $\alpha+2 > \frac{l(\alpha+1)}{l-k}$. It follows from the following inequality:
\begin{align*}
\alpha+2-\frac{l(\alpha+1)}{l-k}=\frac{l-k(\alpha+2)}{l-k} \ge \frac{2(\alpha+2)^{2}-k(\alpha+2)}{l-k} >0. 
\end{align*}
Then the inequality $(1)$ follows from Lemma \ref{lem:general} so that $N$ is ample. \

The argument in the proof of Lemma \ref{lem:0 Picard number one} gives the global generation of $M$, so we may assume that $k \ge 1$. Suppose that $M$ is not $k$-very ample. Then Theorem \ref{thm:criterion} implies that there exists an effective divisor $D$ on ${{\rm Bl}_{r}}(S)$ such that $N.D-k-1 \le D^{2}<\frac{N.D}{2} <k+1$. If $D=\sum a_{i}E_{i}$ with $a_{i} \ge 0$, we have $(\alpha+1) \sum a_{i}-k-1=N.D-k-1 \le D^{2}=-\sum a_{i}^{2}$, that is, $(\alpha+2)\sum a_{i} \le k+1$. Since $\alpha \ge k$, it is a contradiction. \

So we may assume that $D=\pi^{*}D_{S}-\sum a_{i}E_{i}$ with $D_{S}=aL$ for some $a \in \mathbb{N}$. So the inquality $(1)$ implies that 
\begin{align*}
(\pi^{*}L-\frac{l(\alpha+1)}{l-k} \cdot \sum E_{i}.D)=(L.D_{S})-\frac{l(\alpha+1)}{l-k} \sum a_{i}>0,
\end{align*}
i.e. $(\alpha+1)\sum a_{i}<\frac{(L.D_{S})(l-k)}{l}$. Then $2k+1 \ge (N.D)=(L.D_{S})-(\alpha+1)\sum a_{i}>\frac{k(L.D_{S})}{l}=ak$. So there are two possibilities:
\begin{enumerate} [(i)]
\item $a=1$, i.e. $D_{S}=L$ and $k<N.D < 2k+2$.
\item $a=2$, i.e. $D_{S}=2L$ and $N.D=2k+1$.
\end{enumerate}

For the case ${\rm (i)}$, consider the inquality $D^{2}=l-\sum a_{i}^{2}<\frac{1}{2}(N.D)=\frac{1}{2}(l-(\alpha+1)\sum a_{i})$ so that 
\begin{align*}
(\alpha+1) \sum a_{i}<-l+2\sum a_{i}^{2}.
\end{align*}
Since $k<(N.D)$, we have $-1 <(N.D)-k-1 \le D^{2}=l-\sum a_{i}^{2}$, i.e. $\sum a_{i}^{2} \le l$. Moreover, the inequality $l-(\alpha+1)\sum a_{i}=N.D<2k+2$ gives
\begin{align*}
l<2(k+1)+(\alpha+1)\sum a_{i}.
\end{align*}
By summing up the above inequalities, we have:
\begin{align*}
l^{2} &<4(k+1)^{2}+4(k+1)(\alpha+1)\sum a_{i}+(\alpha+1)^{2}(\sum a_{i})^{2} \\
&<4(k+1)^{2}+4(k+1)(-l+2 \sum a_{i}^{2})+(\alpha+1)^{2}r(\sum a_{i}^{2}) \\
&\le 4(k+1)^{2}+4(k+1)l+\frac{(\alpha+1)^{2}}{(\alpha+2)^{2}}l^{2}.
\end{align*}
So $(4\alpha+6)l \le \frac{2\alpha+3}{(\alpha+2)^{2}}l^{2}<4(k+1)^{2}+4(k+1)l$. It implies the following:
\begin{align*}
4(k+2)^{2} \le 4(\alpha+2)^{2} \le 2l \le (4\alpha-4k+2)l <4(k+1)^{2},
\end{align*}
which is a contradiction. Thus we exclude the case ${\rm (i)}$. \

For the case ${\rm (ii)}$, consider the inequalities $2l-(\alpha+1)\sum a_{i}=(N.D)=2k+1$ and $4l-\sum a_{i}^{2}=D^{2}=k$ so that $(\alpha+1)\sum a_{i}=2l-2k-1$ and $\sum a_{i}^{2}=4l-k$. Again, from $(N.D)<2k+2$, we have:
\begin{align*}
4l^{2} &<4(k+1)^{2}+4(k+1)(\alpha+1) \sum a_{i}+(\alpha+1)^{2}(\sum a_{i})^{2} \\
&\le 4(k+1)^{2}+4(k+1)(2l-2k-1)+(\alpha+1)^{2}(4l-k) \cdot \frac{l}{(\alpha+2)^{2}} \\
&=-4k^{2}-4k+8kl+8l-\frac{(\alpha+1)^{2}}{(\alpha+2)^{2}}kl+\frac{(\alpha+1)^{2}}{(\alpha+2)^{2}}4l^{2} \\
&<(8k+8)l+\frac{(\alpha+1)^{2}}{(\alpha+2)^{2}}4l^{2}.
\end{align*}
Thus we have $\frac{2\alpha+3}{(\alpha+2)^{2}}l<2k+2$, i.e. $4\alpha+6<2k+2$, which contradicts our assumption on $\alpha$. 

Finally, the similar argument in the proof of Theorem \ref{thm:higher Picard number one} yields the desired result for an arbitrary $c \in \mathbb{N}$, so we are done.  
\end{proof}

\end{subsection}

\end{section}

\begin{section}{Positivity on general blow-ups of abelian surfaces.} \label{section:4}

In this section, we show how \cite[Theorem 1]{ST1} can be extended to higher order embeddings of line bundles of the form $c \cdot \pi^{*}L-\alpha \cdot \sum_{i=1}^{r}E_{i}$ for a non-negative integer $\alpha$ and $c \in \mathbb{N}$. 

\begin{thm} \label{thm:any Picard number}
Let $(S,L)$ be a polarized abelian surface of type $(1,d)$. Let $\pi:{\rm Bl}_{r}(S) \rightarrow S$ be the blow-up of $S$ at $2 \le r < \frac{2d}{(\ceil{\frac{\alpha}{c}}+1)^{2}}-2$ general points. 
\begin{enumerate}[(1)]
\item Assume that $S$ contains no elliptic curve of degree $\le \ceil{\frac{\alpha}{c}} +1$ with respect to $L$. Then the line bundle $M_{c,\alpha}=c \cdot \pi^{*}L_{S}-\alpha \cdot \sum_{i=1}^{r} E_{i}$ is globally generated when $\alpha \ge 0$.
\item Assume that $S$ contains no elliptic curve of degree $\le \ceil{\frac{\alpha}{c}} +2$ with respect to $L$. Then the line bundle $M_{c,\alpha}=c \cdot \pi^{*}L_{S}-\alpha \cdot \sum_{i=1}^{r} E_{i}$ is very ample when $\alpha \ge 1$. 
\end{enumerate}
\end{thm}

\begin{proof}

$(1)$ We deal with the case $c=1$ first. If $\alpha$=0, $M_{c,\alpha}=\pi^*L_S$. Since we obtain $L^2=2d\ge 6$ from the asusmption $2 < \frac{2d}{(\alpha+1)^2}-2$, and there is no elliptic curve $E$ on $S$ satisfying the inequality $(L.E)=1$, the global generation of $L_S$ follows from \cite[Theorem 1.1]{T'}. Therefore $M_{c,\alpha}=\pi^*L_S$ is globally generated. So we may assume that $\alpha \ge 1$.
Let $N=M_{c,\alpha}-K_{X}=\pi^{*}L-(\alpha+1) \cdot \sum_{i=1}^{r}E_{i}$. Assume that $S$ contains no elliptic curve of degree $\le \alpha+1$ with respect to $L$. We claim that $\epsilon(L;r)>\alpha+1$. From Proposition \ref{prop:multi lower bound} and Proposition \ref{prop:one lower bound}, it is sufficient to check that $\epsilon_{0}>\alpha+1$ and $\frac{\sqrt{2d(r-1)}}{r}>\alpha+1$, where $\epsilon_{0}$ is defined as in Propositon \ref{prop:one lower bound}. The first one is obtained directly by the assumption on $S$, so we focus on the second one. However, it is also immediate since the assumption on $r$ implies that 
\begin{align*}
\frac{\sqrt{2d(r-1)}}{r} > \frac{(\alpha+1) \sqrt{(r-1)(r+2)}}{r} \ge \alpha+1.
\end{align*}
Thus we have $\epsilon(L;r)>\alpha+1$. By Lemma \ref{lem:general}, the generality condition gives that $N$ is ample on ${\rm Bl}_{r}(S)$. Moreover, we have $N^{2} =2d-r(\alpha+1)^{2} \ge 2(\alpha+1)^{2} \ge 5$. Applying Theorem \ref{thm:Reider}, we are done once we exclude the following two cases: there exists an effective divisor $D$ on ${\rm Bl}_{r}(S)$ such that
\begin{enumerate}[i)]
\item $D.N=0$, $D^{2}=-1$. 
\item $D.N=1$, $D^{2}=0$.
\end{enumerate}

The case ${\rm i)}$ is excluded by the ampleness of $N$, so consider the case ${\rm ii)}$. Since $D^{2}=0$, we may assume that $D=\pi^{*}D_{S}-\sum m_{i}E_{i}$, where $D_{S}$ is an effective divisor on $S$. By Hodge index theorem, 
\begin{align*}
2d \sum m_{i}^{2}=L^{2} \cdot D_{S}^{2} \le (L.D_{S})^{2}&=1+(\alpha+1)^{2}(\sum m_{i})^{2}+2(\alpha+1)\sum m_{i} \\
&\le 1+(2d-2(\alpha+1)^{2})\sum m_{i}^{2}+2(\alpha+1)\sum m_{i}^{2},
\end{align*}
i.e. $2(\alpha+1)\alpha\sum m_{i}^{2} \le 1$. This gives that $\sum m_{i}^{2}=0$. Hence $D_{S}$ is an elliptic curve of degree $1$ with respect to $L$, which contradicts our assumption on $S$. 

For an arbitrary $c\in \mathbb{N}$, write $\alpha = c\floor{\frac{\alpha}{c}} + \alpha'$ where $0\le \alpha' < c$. Then $\alpha = (c-\alpha')\floor{\frac{\alpha}{c}}+\alpha'\ceil{\frac{\alpha}{c}}$, so we can write 
\begin{align*}
M_{c,\alpha}=(\pi^*L_S-\floor{\frac{\alpha}{c}}\cdot\sum_{i=1}^{r}E_{i})^{\otimes c-\alpha'}\otimes (\pi^*L_S-\ceil{\frac{\alpha}{c}}\cdot\sum_{i=1}^{r}E_{i})^{\otimes \alpha'}.
\end{align*}
By the proof of $c=1$ case, $\pi^*L_S-\floor{\frac{\alpha}{c}}\cdot\sum_{i=1}^{r}E_{i}$ and $\pi^*L_S-\ceil{\frac{\alpha}{c}}\cdot\sum_{i=1}^{r}E_{i}$ are globally generated, i.e. $0$-very ample. Thus, by \cite[Theorem 1.1]{HTT}, we obtain that $M_{c,\alpha}$ is globally generated.


$(2)$ Next, assume that $S$ contains no elliptic curve of degree $\le \alpha+2$ with respect to $L$. As before, we deal with $c=1$ case first. For $\alpha=1$, it is just \cite[Theorem 1]{ST1}. So we may assume that $\alpha \ge 2$. Let $N$ be as above. Then the proof of $(1)$ shows that $N$ is ample with $N^{2}>2(\alpha+1)^{2} \ge 10$. Applying Reider's theorem (Theorem \ref{thm:Reider}), we need to exclude the existence of an effective divisor $D$ on ${\rm Bl}_{r}(S)$ satisfying one of the following cases:
\begin{enumerate}[i)]
\item $D.N=0$, $D^{2}=-1$ or $-2$. 
\item $D.N=1$, $D^{2}=0$ or $-1$. 
\item $D.N=2$, $D^{2}=0$.
\end{enumerate}

By the ampleness of $N$, the case ${\rm i)}$ can be excluded. First, we claim that any such $D$, if exists, cannot contain some of the exceptional divisors. Suppose that $D=D'+aE$, where $a \ge 1$, $E=E_{i}$ for some $i$, and $D'$ does not contain $E$. Then 
\begin{align*}
2 \ge D.N=D'.N+a(\alpha+1),
\end{align*}
which is a contradiction. Thus we may assume that $D=\pi^{*}D_{S}-\sum m_{i}E_{i}$, where $D_{S}$ is an effective divisor on $S$ and $m_{i} \ge 0$ for all $i$. \

The argument in the proof of $(1)$ eliminates the existence of $D$ satisfying $D.N=1$ and $D^{2}=0$, so it remains to deal with the last two cases. First, suppose the case when $D.N=1$ and $D^{2}=-1$. It is immediate that 
\begin{align*}
(L.D_{S})=1+(\alpha+1)\sum m_{i} \text{ and } D_{S}^{2}=-1+\sum m_{i}^{2}.
\end{align*}
By Hodge index theorem, we have
\begin{align*}
2d(-1+\sum m_{i}^{2})=L^{2} \cdot D_{S}^{2} \le (L.D_{S})^{2} &= 1+2(\alpha+1) \sum m_{i}+(\alpha+1)^{2}(\sum m_{i})^{2} \\
&< 1+2(\alpha+1) \sum m_{i}^{2}+(2d-2(\alpha+1)^{2})\sum m_{i}^{2},
\end{align*}
that is, $2\alpha(\alpha+1) \sum m_{i}^{2}<2d+1$. Now, we claim that $\sum m_{i}^{2} \ge 5$. Since $D_{S}$ is an effective divisor on $S$, $D_{S}^{2} \ge 0$, i.e. $\sum m_{i}^{2} \ge 1$. Moreover, by adjunction formula, $D_{S}^{2}$ should be an even number. So if $\sum m_{i}^{2}<5$, it should be either $1$ or $3$. If $\sum m_{i}^{2}=1$, then $D_{S}$ is an elliptic curve of degree $\alpha+2$ with respect to $L$, which contradicts our assumption on $S$. Next, if $\sum m_{i}^{2}=3$, $D_{S}^{2}=2$. Then $D_{S}$ is a principal polarization(,i.e. of type $(1,1)$) passing through $3$ general points in $S$. However, three points on a principal polarization are not general so that $\sum m_{i}^{2} \ge 5$. \

Now, we claim that $L-(\alpha+1)D_{S}$ is ample. By \cite[Lemma 4.3.2]{BL}, it is sufficient to show that $(L-(\alpha+1)D_{S}.L)>0$ and $(L-(\alpha+1)D_{S})^{2}>0$. Note that 
\begin{align*}
(L-(\alpha+1)D_{S}.L)&=2d-(\alpha+1)(1+(\alpha+1)\sum m_{i}) \\
&>2\alpha(\alpha+1) \sum m_{i}^{2}-1-(\alpha+1)^{2}\sum m_{i}-(\alpha+1) \\
&\ge (\alpha^{2}-1) \sum m_{i}^{2}-(\alpha+2) \\
&>0.
\end{align*}
For the inequality $(L-(\alpha+1)D_{S})^{2}>0$, we claim that either $\sum m_{i}^{2}>5$ or $\sum m_{i}^{2}>\sum m_{i}$ holds. Suppose that $\sum m_{i}^{2} \le 5$. Since $\sum m_{i}^{2} \ge 5$, it means that $\sum m_{i}^{2}=5$. Then there are two possibilities; first one is $m_{1}=2$, $m_{2}=1$, $m_{3}=\dots=m_{r}=0$, and the second one is $m_{1}=\dots=m_{5}=1$, $m_{6}=\dots=m_{r}=0$ by renumbering if necessary. The first one implies $\sum m_{i}^{2}> \sum m_{i}$, so we need to exclude the second one. In this case, since $D_{S}^{2}=4$, it is a polarization of type $(1,2)$ passing through $5$ general points. However, it is impossible by the generality assumption, which proves the claim. Now, consider the following inequalities:
\begin{align*}
(L-(\alpha+1)D_{S})^{2}&=2d+(\alpha+1)^{2}(-1+\sum m_{i}^{2})-2(\alpha+1)(1+(\alpha+1)\sum m_{i}) \\
&>2\alpha(\alpha+1)\sum m_{i}^{2}-1+(\alpha+1)^{2}\sum m_{i}^{2}-2(\alpha+1)^{2} \sum m_{i}-\alpha^{2}-4\alpha -3 \\
&\ge (\alpha^{2}-1) \sum m_{i}^{2}-\alpha^{2}-4\alpha-4 \\
&\ge 5(\alpha^{2}-1)-\alpha^{2}-4\alpha-4 \\
&=4\alpha^{2}-4\alpha-9 \\
&\ge -1. 
\end{align*}
However, the above claim implies that at least one of the second and the third inequalities is strict. Hence $(L-(\alpha+1)D_{S})^{2}>0$, i.e. $L-(\alpha+1)D_{S}$ is ample. 

So we have the following inequalities:
\begin{align*}
0<(L-(\alpha+1)D_{S}.D_{S})&=1+(\alpha+1)\sum m_{i}-(\alpha+1)(-1+\sum m_{i}^{2}) \\
&=-(\alpha+1) \sum m_{i}^{2}+(\alpha+1)\sum m_{i}+\alpha+2 \\
&=(\alpha+2)+(\alpha+1)(\sum m_{i}-\sum m_{i}^{2}).
\end{align*}
Thus all $m_{i}$ are either $0$ or $1$. Let $\beta$ be the number of $1$'s. Then we have 
\begin{align*}
2\alpha(\alpha+1) \beta < 2d+1 \text{, } L.D_{S}=1+(\alpha+1)\beta \text{, and } D_{S}^{2}=-1+\beta.
\end{align*}
Again, Hodge index theorem gives 
\begin{align*}
(\beta-1)((2\alpha^{2}+2\alpha)\beta-1)<L^{2} \cdot D_{S}^{2} \le (D_{S}.L)^{2}=(1+(\alpha+1)\beta)^{2},
\end{align*}
that is, $(\alpha^{2}-1)\beta<2\alpha^{2}+4\alpha+3$. Since $m_{i}$'s are either $0$ or $1$, $\beta=\sum m_{i}^{2} \ge 7$. So $7(\alpha^{2}-1) \le \beta(\alpha^{2}-1)<2\alpha^{2}+4\alpha+3$, i.e. $5\alpha^{2}-4\alpha-10<0$, which is a contradiction since $\alpha \ge 2$. Therefore, we can exclude the case when $D.N=1$ and $D^{2}=-1$. \

Finally, suppose that $D.N=2$ and $D^{2}=0$. Then the Hodge index theorem gives that
\begin{align*}
2d\sum m_{i}^{2} \le (2+(\alpha+1)\sum m_{i})^{2}&=4+4(\alpha+1)\sum m_{i}+(\alpha+1)^{2}(\sum m_{i})^{2} \\
&<4+4(\alpha+1) \sum m_{i}^{2}+(2d-2(\alpha+1)^{2})\sum m_{i}^{2},
\end{align*}
i.e. $(\alpha^{2}-1) \cdot \sum m_{i}^{2}<2$. Since $\alpha \ge 2$, we have $\sum m_{i}^{2}=0$. Then $D_{S}$ is an elliptic curve of degree $2$ with respect to $L$, which is a contradiction. So $\pi^{*}L-\alpha \cdot \sum_{i=1}^{r}E_{i}$ is very ample as desired. 

For an arbitrary $c\in \mathbb{N}$, write $\alpha = c\floor{\frac{\alpha}{c}} + \alpha'$, where $0\le \alpha' < c$. If $\alpha \ge c$, then $\floor{\frac{\alpha}{c}}\ge 1$.
Thus $\alpha = (c-\alpha')\floor{\frac{\alpha}{c}}+\alpha'\ceil{\frac{\alpha}{c}}$, so we can write 
\begin{align*}
M_{c,\alpha}=(\pi^*L_S-\floor{\frac{\alpha}{c}}\cdot\sum_{i=1}^{r}E_{i})^{\otimes c-\alpha'}\otimes (\pi^*L_S-\ceil{\frac{\alpha}{c}}\cdot\sum_{i=1}^{r}E_{i})^{\otimes \alpha'}.
\end{align*}
By the proof of $c=1$ case, $\pi^*L_S-\floor{\frac{\alpha}{c}}\cdot\sum_{i=1}^{r}E_{i}$ and $\pi^*L_S-\ceil{\frac{\alpha}{c}}\cdot\sum_{i=1}^{r}E_{i}$ are very ample, i.e. $1$-very ample. Thus, by \cite[Theorem 1.1]{HTT}, $M_{c,\alpha}$ is $c$-very ample, therefore it is very ample.

If $1 \le \alpha < c$, then $\alpha'>0$. In this case, $\pi^*L_S-\floor{\frac{\alpha}{c}}\cdot\sum_{i=1}^{r}E_{i}$ is at least globally generated, i.e. $0$-very ample, and $\pi^*L_S-\ceil{\frac{\alpha}{c}}\cdot\sum_{i=1}^{r}E_{i}$ is very ample, i.e. $1$-very ample. Therefore, by \cite[Theorem 1.1]{HTT}, $M_{c,\alpha}$ is $\alpha'$-very ample so that it is very ample.
\end{proof}

If $S$ is a simple abelian surface, then the condition for elliptic curves is not needed. On most abelian surfaces, however, this may be a nuisance. Regarding this, the following corollary tells us that the condition is redundant if we choose a large $d$ in the case of a very general abelian surface (in the sense of \cite[Theorem 1.(b)]{B}). 

\begin{cor} \label{cor:application 1}
Let $(S,L)$ be a very general polarized abelian surface of type $(1,d)$. Let $\pi:{\rm Bl}_{r}(S) \rightarrow S$ be the blow-up of $S$ at $2 \le r < \frac{2d}{(\ceil{\frac{\alpha}{c}}+1)^{2}}-2$ general points.  
\begin{enumerate}[(1)] 
\item Assume that $d>\frac{81}{16}(\ceil{\frac{\alpha}{c}}+1)^{2}$. Then the line bundle $c \cdot \pi^{*}L-\alpha \cdot \sum_{i=1}^{r} E_{i}$ is globally generated when $\alpha \ge 0$.
\item Assume that $d>\frac{81}{16}(\ceil{\frac{\alpha}{c}}+2)^{2}$. Then the line bundle $c \cdot \pi^{*}L-\alpha \cdot \sum_{i=1}^{r} E_{i}$ is very ample when $\alpha \ge 1$. 
\end{enumerate}
\end{cor}

\begin{proof}
When $\alpha=0$, the global generation of a given line bundle is clear by \cite[Theorem 1.1]{T'}. So we may assume that $\alpha \ge 1$. By \cite[Theorem 4.5]{S}, the very generality of $(S,L)$ and the inequality $d>\frac{81}{16}(\ceil{\frac{\alpha}{c}}+1)^{2}$ give the $N_{\ceil{\frac{\alpha}{c}}-1}$ property of $L$. It implies that $S$ contains no elliptic curve of degree $\le \ceil{\frac{\alpha}{c}}+1$ with respect to $L$ by \cite[Theorem 1.1]{KL1509}. Now the first statement is a direct consequence of Theorem \ref{thm:any Picard number}-(1). \

The second statement can be proven similarly. 
\end{proof}

On a polarized abelian surface $(S,L)$, there are many results about the lower bound of its Seshadri constant (e.g., \cite{B, BDHKKSS, BS1, L, N}). So it is sometimes easier to deal with the Seshadri constant than the degree of an elliptic curve on $(S,L)$. The following result says that we can change from the condition for elliptic curves in Theorem \ref{thm:any Picard number} to that for the Seshadri constant. 

\begin{cor} \label{cor:application 2}
Let $(S,L)$ be a polarized abelian surface of type $(1,d)$. Let $\pi:{\rm Bl}_{r}(S) \rightarrow S$ be the blow-up of $S$ at $2 \le r < \frac{2d}{(\ceil{\frac{\alpha}{c}}+1)^{2}}-2$ general points. 
\begin{enumerate}[(1)]
\item Assume that $\epsilon(L;1)>\frac{5-\sqrt{5}}{2}(\ceil{\frac{\alpha}{c}}+1)$ and $d \ge \frac{5}{2}(\ceil{\frac{\alpha}{c}}+1)^{2}$. Then the line bundle $c \cdot \pi^{*}L_{S}-\alpha \cdot \sum_{i=1}^{r} E_{i}$ is globally generated when $\alpha \ge 0$.
\item Assume that $\epsilon(L;1)>\frac{5-\sqrt{5}}{2}(\ceil{\frac{\alpha}{c}}+2)$ and $d \ge \frac{5}{2}(\ceil{\frac{\alpha}{c}}+2)^{2}$. Then the line bundle $c \cdot \pi^{*}L_{S}-\alpha \cdot \sum_{i=1}^{r} E_{i}$ is very ample when $\alpha \ge 1$. 
\end{enumerate}
\end{cor}

\begin{proof}
Since the proofs of $(1)$ and $(2)$ are similar, we only give the proof of $(2)$. By Theorem \ref{thm:any Picard number}, it is sufficient to show that there exists no elliptic curve of degree $\le \ceil{\frac{\alpha}{c}}+2$ with respect to $L$. Suppose not. Since $L^{2}=2d \ge 5(\ceil{\frac{\alpha}{c}}+2)^{2}$, $L$ does not satisfy property $N_{\ceil{\frac{\alpha}{c}}}$ by \cite[Theorem 1.1]{KL1509}. By \cite[Theorem 1.1]{S}, we obtain:
\begin{align*}
2d \cdot (\epsilon(L;1)-\ceil{\frac{\alpha}{c}}-2)-(\ceil{\frac{\alpha}{c}}+2) \cdot \epsilon(L;1)^{2} \le 0.
\end{align*}
Note that this inequality is equivalent to the following one:
\begin{align*}
\epsilon(L;1)^{2}-(\frac{2d}{\ceil{\frac{\alpha}{c}}+2}) \cdot \epsilon(L;1)+2d \ge 0.
\end{align*}
It implies that either $\epsilon(L;1) \le \frac{d-\sqrt{d^{2}-2d(\ceil{\frac{\alpha}{c}}+2)^{2}}}{\ceil{\frac{\alpha}{c}}+2}$ or $\epsilon(L;1) \ge \frac{d+\sqrt{d^{2}-2d(\ceil{\frac{\alpha}{c}}+2)^{2}}}{\ceil{\frac{\alpha}{c}}+2}$ holds. \

Now, we claim that $\sqrt{2d}<\frac{d+\sqrt{d^{2}-2d(\ceil{\frac{\alpha}{c}}+2)^{2}}}{\ceil{\frac{\alpha}{c}}+2}$. For this claim, it is sufficient to show that 
\begin{align*}
\sqrt{2} \cdot (\ceil{\frac{\alpha}{c}}+2) < \sqrt{d}+\sqrt{d-2(\ceil{\frac{\alpha}{c}}+2)^{2}},
\end{align*}
which is clear by our assumption on $d$. However, since $\epsilon(L;1) \le \sqrt{2d}$, the inequality $\epsilon(L;1) \ge \frac{d+\sqrt{d^{2}-2d(\ceil{\frac{\alpha}{c}}+2)^{2}}}{\ceil{\frac{\alpha}{c}}+2}$ can be excluded. Moreover, it is easy to see that 
\begin{align*}
\frac{d-\sqrt{d^{2}-2d(\ceil{\frac{\alpha}{c}}+2)^{2}}}{\ceil{\frac{\alpha}{c}}+2} \le \frac{5-\sqrt{5}}{2}(\ceil{\frac{\alpha}{c}}+2)
\end{align*}
since $d \ge \frac{5}{2}(\ceil{\frac{\alpha}{c}}+2)^{2}$. However, since $\epsilon(L;1) >\frac{5-\sqrt{5}}{2}(\ceil{\frac{\alpha}{c}}+2)$, it is a contradiction. Hence, there exists no elliptic curve of degree $\le \ceil{\frac{\alpha}{c}}+2$ with respect to $L$, so the conclusion follows. 
\end{proof}

From Theorem \ref{thm:any Picard number}, we present the following question:

\begin{ques} \label{ques:abelian surface with any Picard number}
Let $(S,L)$ be a polarized abelian surface of type $(1,d)$. Let $\pi:{\rm Bl}_{r}(S) \rightarrow S$ be the blow-up of $S$ at $2 \le r < \frac{2d}{(\ceil{\frac{\alpha}{c}}+1)^{2}}-2$ general points for integers $\alpha \ge k \ge 0$ and $c \in \mathbb{N}$. Assume that $S$ contains no elliptic curve of degree $\le \ceil{\frac{\alpha}{c}}+\ceil{\frac{k}{c}}+1$ with respect to $L$. Then is the line bundle $c \cdot \pi^{*}L-\alpha \cdot \sum_{i=1}^{r}E_{i}$ $k$-very ample?
\end{ques}

As a result of our effort, we provide a partial answer to Question \ref{ques:abelian surface with any Picard number}.

\begin{thm} \label{thm:k-very ampleness on abelian surface}
Let $(S,L)$ be a polarized abelian surface of type $(1,d)$. Let $\pi:{\rm Bl}_{r}(S) \rightarrow S$ be the blow-up of $S$ at $2 \le r <\frac{2d}{(\ceil{\frac{\alpha}{c}}+\ceil{\frac{k}{c}}+1)^{2}}-2$ general points for integers $\alpha \ge k \ge 0$. Assume that $S$ contains no elliptic curve of degree $\le \ceil{\frac{\alpha}{c}}+\ceil{\frac{k}{c}}+1$ with respect to $L$. Then $M=c\cdot\pi^{*}L-\alpha \cdot \sum_{i=1}^{r}E_{i}$ is $k$-very ample. 
\end{thm}

\begin{lem}\label{key lemma}
Let $N \equiv \pi^{*}L-(\alpha+1) \sum_{i=1}^{r}E_{i}$. Then there exists no effective divisor $D$ on ${\rm Bl}_{r}(S)$ satisfying the inequality:
\begin{align*}
(N.D)-k-1 \le D^{2}<\frac{1}{2}(N.D)<k+1.
\end{align*}
\end{lem}

\begin{proof}
Suppose not. Note that $N$ is ample with $N^{2} > 4k+5$ by the proof of Theorem \ref{thm:any Picard number}. If $D=D'+aE$ such that $a \ge 1$, $E=E_{i}$ for some $i$, and $D'$ does not contain $E$, then 
\begin{align*}
2k+1 \ge D.N=D'.N+a(\alpha+1) \ge D'.N+a(k+1).
\end{align*}
So we may assume that $a=1$. In this case, 
\begin{align*}
D^{2} \ge N.D-k-1 \ge D'.N+k+1-k-1 \ge 1
\end{align*}
so that $D^{2}=D'^{2}-1 \ge 1$, that is, $D'^{2} \ge 2$. By Hodge index theorem, 
\begin{align*}
2k+1 \ge D'.N+k+1 \ge \sqrt{D'^{2}} \cdot \sqrt{N^{2}}+k+1>\sqrt{2} \cdot \sqrt{2}(k+1)+k+1=3k+3,
\end{align*}
which is a contradiction. Thus we may assume that $D=\pi^{*}D_{S}-\sum m_{i}E_{i}$, where $D_{S}$ is an effective divisor on $S$ and $m_{i} \ge 0$ for all $i$. If $D^{2} \ge 1$, again by Hodge index theorem, 
\begin{align*}
(2k+1)^{2} \ge (D.N)^{2} \ge D^{2} \cdot N^{2} \ge 2d-r(\alpha+1)^{2} &\ge (r+2)(\alpha+k+1)^{2}-r(\alpha+1)^{2} \\
&=rk(2\alpha+2+k)+2(\alpha+k+1)^{2},
\end{align*}
which is a contradiction. So $D^{2} \le 0$. 

Now, we use the induction on $k$. For $k=0, 1$, we are done by Theorem \ref{thm:any Picard number}. Suppose it holds for $k-1 \ge 1$. If our $D$ satisfies $N.D-k \le D^{2}$, then the induction hypothesis implies that $N.D \ge 2k$. However, it means $k \le N.D-k \le D^{2} \le 0$, which is a contradiction. So $D^{2}=N.D-k-1$ should hold. 

Moreover, since $\alpha+k+1=(\alpha+1)+(k-1)+1$, the induction hypothesis gives the $(k-1)$-very ampleness of $N \equiv \pi^{*}L-(\alpha+1) \cdot \sum E_{i}$. Since $D$ is effective, $N.D \ge k-1$. Combining it with $D^{2} \le 0$, there exist only three possible cases:
\begin{enumerate}[i)]
\item $D^{2}=-2$ and $N.D=k-1$. 
\item $D^{2}=-1$ and $N.D=k$.
\item $D^{2}=0$ and $N.D=k+1$. 
\end{enumerate}

We deal with the case iii) first. Note that $D_{S}^{2}=\sum m_{i}^{2}$ and $L.D_{S}=k+1+(\alpha+1) \cdot \sum m_{i}$. By Hodge index theorem, 
\begin{align*}
2d(\sum m_{i}^{2}) \le (L.D_{S})^{2}&=(\alpha+1)^{2}(\sum m_{i})^{2}+2(k+1)(\alpha+1) \cdot \sum m_{i}+(k+1)^{2} \\
&<(k+1)^{2}+2(k+1)(\alpha+1)\sum m_{i}^{2}+(2d-(2\alpha+k+1)^{2})\sum m_{i}^{2}. 
\end{align*}
Hence $(2(\alpha+k+1)^{2}-2(k+1)(\alpha+1))\sum m_{i}^{2}<(k+1)^{2}$ holds. However, since $\sum m_{i}^{2} \ge 1$, this is impossible. 

Next, consider the case ii). Note that $D_{S}^{2}=\sum m_{i}^{2}-1$ and $L.D_{S}=k+(\alpha+1) \cdot \sum m_{i}$. Then Hodge index theorem gives 
\begin{align*}
2d(-1+\sum m_{i}^{2}) \le (L.D_{S})^{2}&=k^{2}+2k(\alpha+1)\sum m_{i}+(\alpha+1)^{2}(\sum m_{i})^{2} \\
&< k^{2}+2k(\alpha+1) \sum m_{i} +(2d-2(\alpha+k+1)^{2})\sum m_{i}^{2}. 
\end{align*}
So we have $2d > 2(\alpha+k+1)^{2} \sum m_{i}^{2}-2k(\alpha+1) \sum m_{i}-k^{2}$. Suppose that there exists $m_{i} \ge 2$. Then $\sum m_{i}^{2} \ge \sum m_{i}+3$ so that
\begin{align*}
2d &>2(\alpha+k+1)^{2} \sum m_{i}^{2}-2k(\alpha+1)\sum m_{i}-k^{2} \\
&\ge (2(\alpha+k+1)^{2}-2k(\alpha+1))\sum m_{i}+6(\alpha+k+1)^{2}-k^{2}.
\end{align*}
Now, we claim that $L-(\alpha+1)D_{S}$ is ample. As in the proof of Theorem \ref{thm:any Picard number}, it is sufficient to show that $(L-(\alpha+1)D_{S}.L)>0$ and $(L-(\alpha+1)D_{S})^{2} >0$. Note that 
\begin{align*}
(L-(\alpha+1)D_{S}.L)&=2d-(\alpha+1)(k+(\alpha+1)\sum m_{i}) \\
&\ge (2(\alpha+k+1)^{2}-2k(\alpha+1)-(\alpha+1)^{2})\sum m_{i}+6(\alpha+k+1)^{2}-k^{2}-k(\alpha+1) \\
&>0 
\end{align*}
and 
\begin{align*}
(L-(\alpha+1)D_{S})^{2}&=2d+(\alpha+1)^{2}(-1+\sum m_{i}^{2})-2(\alpha+1)(k+(\alpha+1)\sum m_{i}) \\
&\ge 2d+(\alpha+1)^{2}(2+\sum m_{i})-2(\alpha+1)(k+(\alpha+1)\sum m_{i}) \\
&>0.
\end{align*}
Hence $L-(\alpha+1)D_{S}$ is ample. However, we have:
\begin{align*}
0<(L-(\alpha+1)D_{S}.D_{S})&=k+(\alpha+1)\sum m_{i}-(\alpha+1)(-1+\sum m_{i}^{2}) \\
&=(\alpha+1)(\sum m_{i}-\sum m_{i}^{2})+\alpha+k+1 \\
& \le -3(\alpha+1)+\alpha+k+1 \\
&<0,
\end{align*}
which is a contradiction. So all $m_{i}$'s are either $0$ or $1$. 

Let $\beta$ be the number of $1$'s among $m_{i}$'s, that is, $\beta=\sum m_{i}=\sum m_{i}^{2}$. Since $D_{S}^{2}$ is even, $\beta$ should be odd. If $\beta=1$, $D_{S}$ is an elliptic curve of degree $\alpha+k+1$ with respect to $L$, which is impossible. So $\beta \ge 3$. However, it means that $D_{S}$ passes through $\beta=D_{S}^{2}+1=2h^{0}(D_{S})+1$ general points, which is also impossible. So the case ii) is excluded. 

Finally, it remains to eliminate the case i). Note that $D_{S}^{2}=-2+\sum m_{i}^{2}$ and $L.D_{S}=k-1+(\alpha+1) \cdot \sum m_{i}$. The similar argument as in the case ii) shows that all $m_{i}$'s are either $0$ or $1$. Let $\beta$ be the number of $1$'s among $m_{i}$'s. Since $D_{S}^{2}$ is even, $\beta$ should be even. If $\beta \ge 4$, the generality condition induces a contradiction as above. So let $\beta=2$, i.e. $D_{S}$ is an elliptic curve passing through $2$ general points on $S$. By Riemann-Roch, $h^{0}(\OO_{S}(D_{S}))=h^{1}(\OO_{S}(D_{S}))$. Consider the following short exact sequence:
\begin{align*}
0 \rightarrow \OO_{S}(-D_{S}) \rightarrow \OO_{S} \rightarrow \OO_{D_{S}} \rightarrow 0.
\end{align*}
Then it induces the long exact sequence:
\begin{align*}
0 \rightarrow H^{0}(\OO_{S}(-D_{S})) \rightarrow H^{0}(\OO_{S}) \overset{\psi}{\rightarrow} H^{0}(\OO_{D_{S}}) \rightarrow H^{1}(\OO_{S}(-D_{S})) \overset{\phi}{\rightarrow} H^{1}(\OO_{S}) \rightarrow H^{1}(\OO_{D_{S}}).
\end{align*}
Since $H^{0}(\OO_{S}(-D_{S}))=0$ and $H^{0}(\OO_{S}) \cong \mathbb{C} \cong H^{0}(\OO_{D_{S}})$, $\psi$ is an isomorphism. Thus $\phi$ is injective. Since $h^{1}(\OO_{S})=2$, we have $h^{0}(\OO_{S}(D_{S}))=h^{1}(\OO_{S}(D_{S}))=h^{1}(\OO_{S}(-D_{S})) \le h^{1}(\OO_{S})=2$, that is, ${\rm dim} \lvert D_{S} \rvert \le 1$. Hence $D_{S}$ cannot pass through $2$ general points on $S$, which is a contradiction. Thus we exclude the case i). 
\end{proof}

\begin{proof}[Proof of Theorem \ref{thm:k-very ampleness on abelian surface}] When $c=1$, it is trivial by Lemma \ref{key lemma} and Theorem \ref{thm:criterion}. For an arbitrary $c$, the same procedure as in the proof of Theorem \ref{thm:higher Picard number one} and of Theorem \ref{thm:any Picard number} leads to the conclusion.
\end{proof}

\end{section}

\bibliographystyle{abbrv}
\bibliography{Library}

\end{document}